\documentclass[10pt,a4paper]{article}
\usepackage{amssymb}
\usepackage{amsmath}
\usepackage{amsthm}
\usepackage[T2A]{fontenc}
\usepackage{graphicx}
\usepackage{tikz}

\newtheorem{theorem}{Theorem}[section]
\newtheorem{proposition}[theorem]{Proposition}
\newtheorem{corollary}[theorem]{Corollary}
\newtheorem{lemma}[theorem]{Lemma}

\theoremstyle{definition}

\newtheorem{remark}[theorem]{Remark}

\newcommand{\fighosts}{%
}

\newcommand{\figcorner}{%
\begin{tikzpicture}[x=3.300mm,y=3.300mm]
  \begin{scope}
    \clip (0,0) rectangle (13,13);
    \fill[black!13] (-20.0,35.6) -- (33.0,-17.4) -- (33.0,-22.6) -- (-20.0,30.4) -- cycle;
  \end{scope}
  \begin{scope}
    \clip (13,0) rectangle (18,-5);
    \fill[black!13] (-20.0,35.6) -- (33.0,-17.4) -- (33.0,-22.6) -- (-20.0,30.4) -- cycle;
  \end{scope}
  \draw[line width=0.5pt] (0,0) rectangle (13,13);
  \draw[line width=0.5pt] (13,0) rectangle (18,-5);
  \node at (3.1,3.1) {$Q$};
  \node at (17.0,-1.0) {$D$};
  \draw[|<->|,line width=0.3pt] (0,-7.4) -- (13,-7.4);
  \node[below] at (6.5,-7.1) {\small $N$};
  \draw[|<->|,line width=0.3pt] (13,-7.4) -- (18,-7.4);
  \node[below] at (15.5,-7.1) {\small $t$};
\end{tikzpicture}%
}

\newcommand{\figsplit}{%
\begin{tikzpicture}
\begin{scope}[xshift=0.000mm,x=2.780mm,y=2.780mm]
  \fill[black!11] (1,3) rectangle (12,14);
  \draw[line width=0.5pt] (1,3) rectangle (12,14);
  \node at (6.5,8.5) {\small $H[B]$};
  \draw[black!35,line width=0.15pt] (0,0) -- (0,3);
  \draw[black!35,line width=0.15pt] (1,0) -- (1,3);
  \draw[black!35,line width=0.15pt] (2,0) -- (2,3);
  \draw[black!35,line width=0.15pt] (3,0) -- (3,3);
  \draw[black!35,line width=0.15pt] (4,0) -- (4,3);
  \draw[black!35,line width=0.15pt] (5,0) -- (5,3);
  \draw[black!35,line width=0.15pt] (6,0) -- (6,3);
  \draw[black!35,line width=0.15pt] (7,0) -- (7,3);
  \draw[black!35,line width=0.15pt] (8,0) -- (8,3);
  \draw[black!35,line width=0.15pt] (9,0) -- (9,3);
  \draw[black!35,line width=0.15pt] (10,0) -- (10,3);
  \draw[black!35,line width=0.15pt] (11,0) -- (11,3);
  \draw[black!35,line width=0.15pt] (12,0) -- (12,3);
  \draw[black!35,line width=0.15pt] (13,0) -- (13,3);
  \draw[black!35,line width=0.15pt] (14,0) -- (14,3);
  \draw[black!35,line width=0.15pt] (0,0) -- (14,0);
  \draw[black!35,line width=0.15pt] (0,1) -- (14,1);
  \draw[black!35,line width=0.15pt] (0,2) -- (14,2);
  \draw[black!35,line width=0.15pt] (0,3) -- (14,3);
  \draw[black!35,line width=0.15pt] (0,3) -- (0,14);
  \draw[black!35,line width=0.15pt] (1,3) -- (1,14);
  \draw[black!35,line width=0.15pt] (0,3) -- (1,3);
  \draw[black!35,line width=0.15pt] (0,4) -- (1,4);
  \draw[black!35,line width=0.15pt] (0,5) -- (1,5);
  \draw[black!35,line width=0.15pt] (0,6) -- (1,6);
  \draw[black!35,line width=0.15pt] (0,7) -- (1,7);
  \draw[black!35,line width=0.15pt] (0,8) -- (1,8);
  \draw[black!35,line width=0.15pt] (0,9) -- (1,9);
  \draw[black!35,line width=0.15pt] (0,10) -- (1,10);
  \draw[black!35,line width=0.15pt] (0,11) -- (1,11);
  \draw[black!35,line width=0.15pt] (0,12) -- (1,12);
  \draw[black!35,line width=0.15pt] (0,13) -- (1,13);
  \draw[black!35,line width=0.15pt] (0,14) -- (1,14);
  \draw[black!35,line width=0.15pt] (12,3) -- (12,14);
  \draw[black!35,line width=0.15pt] (13,3) -- (13,14);
  \draw[black!35,line width=0.15pt] (14,3) -- (14,14);
  \draw[black!35,line width=0.15pt] (12,3) -- (14,3);
  \draw[black!35,line width=0.15pt] (12,4) -- (14,4);
  \draw[black!35,line width=0.15pt] (12,5) -- (14,5);
  \draw[black!35,line width=0.15pt] (12,6) -- (14,6);
  \draw[black!35,line width=0.15pt] (12,7) -- (14,7);
  \draw[black!35,line width=0.15pt] (12,8) -- (14,8);
  \draw[black!35,line width=0.15pt] (12,9) -- (14,9);
  \draw[black!35,line width=0.15pt] (12,10) -- (14,10);
  \draw[black!35,line width=0.15pt] (12,11) -- (14,11);
  \draw[black!35,line width=0.15pt] (12,12) -- (14,12);
  \draw[black!35,line width=0.15pt] (12,13) -- (14,13);
  \draw[black!35,line width=0.15pt] (12,14) -- (14,14);
  \draw[line width=0.5pt] (0,0) rectangle (14,14);
  \fill (0.2657,2.3589) -- (0.1708,2.6710) -- (0.2538,2.4641) -- (0.3376,2.7266) -- (0.4179,2.4487) -- (0.5000,2.7582) -- (0.5821,2.4487) -- (0.6624,2.7266) -- (0.7462,2.4641) -- (0.8292,2.6710) -- (0.7343,2.3589) -- cycle;
  \fill (0.2341,2.2521) rectangle (0.7659,2.3512);
  \fill (0.1546,2.1546) -- (0.8454,2.1546) -- (0.7745,2.2452) -- (0.2255,2.2452) -- cycle;
  \fill (0.1708,2.6710) circle (0.1878mm);
  \fill (0.3376,2.7266) circle (0.1878mm);
  \fill (0.5000,2.7582) circle (0.1878mm);
  \fill (0.6624,2.7266) circle (0.1878mm);
  \fill (0.8292,2.6710) circle (0.1878mm);
  \fill (0.2657,1.3589) -- (0.1708,1.6710) -- (0.2538,1.4641) -- (0.3376,1.7266) -- (0.4179,1.4487) -- (0.5000,1.7582) -- (0.5821,1.4487) -- (0.6624,1.7266) -- (0.7462,1.4641) -- (0.8292,1.6710) -- (0.7343,1.3589) -- cycle;
  \fill (0.2341,1.2521) rectangle (0.7659,1.3512);
  \fill (0.1546,1.1546) -- (0.8454,1.1546) -- (0.7745,1.2452) -- (0.2255,1.2452) -- cycle;
  \fill (0.1708,1.6710) circle (0.1878mm);
  \fill (0.3376,1.7266) circle (0.1878mm);
  \fill (0.5000,1.7582) circle (0.1878mm);
  \fill (0.6624,1.7266) circle (0.1878mm);
  \fill (0.8292,1.6710) circle (0.1878mm);
  \fill (12.2657,0.3589) -- (12.1708,0.6710) -- (12.2538,0.4641) -- (12.3376,0.7266) -- (12.4179,0.4487) -- (12.5000,0.7582) -- (12.5821,0.4487) -- (12.6624,0.7266) -- (12.7462,0.4641) -- (12.8292,0.6710) -- (12.7343,0.3589) -- cycle;
  \fill (12.2341,0.2520) rectangle (12.7659,0.3512);
  \fill (12.1546,0.1546) -- (12.8454,0.1546) -- (12.7745,0.2452) -- (12.2255,0.2452) -- cycle;
  \fill (12.1708,0.6710) circle (0.1878mm);
  \fill (12.3376,0.7266) circle (0.1878mm);
  \fill (12.5000,0.7582) circle (0.1878mm);
  \fill (12.6624,0.7266) circle (0.1878mm);
  \fill (12.8292,0.6710) circle (0.1878mm);
  \fill (13.2657,0.3589) -- (13.1708,0.6710) -- (13.2538,0.4641) -- (13.3376,0.7266) -- (13.4179,0.4487) -- (13.5000,0.7582) -- (13.5821,0.4487) -- (13.6624,0.7266) -- (13.7462,0.4641) -- (13.8292,0.6710) -- (13.7343,0.3589) -- cycle;
  \fill (13.2341,0.2520) rectangle (13.7659,0.3512);
  \fill (13.1546,0.1546) -- (13.8454,0.1546) -- (13.7745,0.2452) -- (13.2255,0.2452) -- cycle;
  \fill (13.1708,0.6710) circle (0.1878mm);
  \fill (13.3376,0.7266) circle (0.1878mm);
  \fill (13.5000,0.7582) circle (0.1878mm);
  \fill (13.6624,0.7266) circle (0.1878mm);
  \fill (13.8292,0.6710) circle (0.1878mm);
  \node[above] at (0.5,14) {\small $L$};
  \node[above] at (13.0,14) {\small $t-L$};
  \node[right] at (14,1.5) {\small $t$};
  \node[below] at (7.0,-0.35) {\small $t=3$};
\end{scope}
\begin{scope}[xshift=46.920mm,x=2.780mm,y=2.780mm]
  \fill[black!11] (2,4) rectangle (13,15);
  \draw[line width=0.5pt] (2,4) rectangle (13,15);
  \node at (7.5,9.5) {\small $H[B]$};
  \draw[black!35,line width=0.15pt] (0,0) -- (0,4);
  \draw[black!35,line width=0.15pt] (1,0) -- (1,4);
  \draw[black!35,line width=0.15pt] (2,0) -- (2,4);
  \draw[black!35,line width=0.15pt] (3,0) -- (3,4);
  \draw[black!35,line width=0.15pt] (4,0) -- (4,4);
  \draw[black!35,line width=0.15pt] (5,0) -- (5,4);
  \draw[black!35,line width=0.15pt] (6,0) -- (6,4);
  \draw[black!35,line width=0.15pt] (7,0) -- (7,4);
  \draw[black!35,line width=0.15pt] (8,0) -- (8,4);
  \draw[black!35,line width=0.15pt] (9,0) -- (9,4);
  \draw[black!35,line width=0.15pt] (10,0) -- (10,4);
  \draw[black!35,line width=0.15pt] (11,0) -- (11,4);
  \draw[black!35,line width=0.15pt] (12,0) -- (12,4);
  \draw[black!35,line width=0.15pt] (13,0) -- (13,4);
  \draw[black!35,line width=0.15pt] (14,0) -- (14,4);
  \draw[black!35,line width=0.15pt] (15,0) -- (15,4);
  \draw[black!35,line width=0.15pt] (0,0) -- (15,0);
  \draw[black!35,line width=0.15pt] (0,1) -- (15,1);
  \draw[black!35,line width=0.15pt] (0,2) -- (15,2);
  \draw[black!35,line width=0.15pt] (0,3) -- (15,3);
  \draw[black!35,line width=0.15pt] (0,4) -- (15,4);
  \draw[black!35,line width=0.15pt] (0,4) -- (0,15);
  \draw[black!35,line width=0.15pt] (1,4) -- (1,15);
  \draw[black!35,line width=0.15pt] (2,4) -- (2,15);
  \draw[black!35,line width=0.15pt] (0,4) -- (2,4);
  \draw[black!35,line width=0.15pt] (0,5) -- (2,5);
  \draw[black!35,line width=0.15pt] (0,6) -- (2,6);
  \draw[black!35,line width=0.15pt] (0,7) -- (2,7);
  \draw[black!35,line width=0.15pt] (0,8) -- (2,8);
  \draw[black!35,line width=0.15pt] (0,9) -- (2,9);
  \draw[black!35,line width=0.15pt] (0,10) -- (2,10);
  \draw[black!35,line width=0.15pt] (0,11) -- (2,11);
  \draw[black!35,line width=0.15pt] (0,12) -- (2,12);
  \draw[black!35,line width=0.15pt] (0,13) -- (2,13);
  \draw[black!35,line width=0.15pt] (0,14) -- (2,14);
  \draw[black!35,line width=0.15pt] (0,15) -- (2,15);
  \draw[black!35,line width=0.15pt] (13,4) -- (13,15);
  \draw[black!35,line width=0.15pt] (14,4) -- (14,15);
  \draw[black!35,line width=0.15pt] (15,4) -- (15,15);
  \draw[black!35,line width=0.15pt] (13,4) -- (15,4);
  \draw[black!35,line width=0.15pt] (13,5) -- (15,5);
  \draw[black!35,line width=0.15pt] (13,6) -- (15,6);
  \draw[black!35,line width=0.15pt] (13,7) -- (15,7);
  \draw[black!35,line width=0.15pt] (13,8) -- (15,8);
  \draw[black!35,line width=0.15pt] (13,9) -- (15,9);
  \draw[black!35,line width=0.15pt] (13,10) -- (15,10);
  \draw[black!35,line width=0.15pt] (13,11) -- (15,11);
  \draw[black!35,line width=0.15pt] (13,12) -- (15,12);
  \draw[black!35,line width=0.15pt] (13,13) -- (15,13);
  \draw[black!35,line width=0.15pt] (13,14) -- (15,14);
  \draw[black!35,line width=0.15pt] (13,15) -- (15,15);
  \draw[line width=0.5pt] (0,0) rectangle (15,15);
  \fill (13.2657,3.3589) -- (13.1708,3.6710) -- (13.2538,3.4641) -- (13.3376,3.7266) -- (13.4179,3.4487) -- (13.5000,3.7582) -- (13.5821,3.4487) -- (13.6624,3.7266) -- (13.7462,3.4641) -- (13.8292,3.6710) -- (13.7343,3.3589) -- cycle;
  \fill (13.2341,3.2521) rectangle (13.7659,3.3512);
  \fill (13.1546,3.1546) -- (13.8454,3.1546) -- (13.7745,3.2452) -- (13.2255,3.2452) -- cycle;
  \fill (13.1708,3.6710) circle (0.1878mm);
  \fill (13.3376,3.7266) circle (0.1878mm);
  \fill (13.5000,3.7582) circle (0.1878mm);
  \fill (13.6624,3.7266) circle (0.1878mm);
  \fill (13.8292,3.6710) circle (0.1878mm);
  \fill (0.2657,2.3589) -- (0.1708,2.6710) -- (0.2538,2.4641) -- (0.3376,2.7266) -- (0.4179,2.4487) -- (0.5000,2.7582) -- (0.5821,2.4487) -- (0.6624,2.7266) -- (0.7462,2.4641) -- (0.8292,2.6710) -- (0.7343,2.3589) -- cycle;
  \fill (0.2341,2.2521) rectangle (0.7659,2.3512);
  \fill (0.1546,2.1546) -- (0.8454,2.1546) -- (0.7745,2.2452) -- (0.2255,2.2452) -- cycle;
  \fill (0.1708,2.6710) circle (0.1878mm);
  \fill (0.3376,2.7266) circle (0.1878mm);
  \fill (0.5000,2.7582) circle (0.1878mm);
  \fill (0.6624,2.7266) circle (0.1878mm);
  \fill (0.8292,2.6710) circle (0.1878mm);
  \fill (1.2657,2.3589) -- (1.1708,2.6710) -- (1.2538,2.4641) -- (1.3376,2.7266) -- (1.4179,2.4487) -- (1.5000,2.7582) -- (1.5821,2.4487) -- (1.6624,2.7266) -- (1.7462,2.4641) -- (1.8292,2.6710) -- (1.7343,2.3589) -- cycle;
  \fill (1.2341,2.2521) rectangle (1.7659,2.3512);
  \fill (1.1546,2.1546) -- (1.8454,2.1546) -- (1.7745,2.2452) -- (1.2255,2.2452) -- cycle;
  \fill (1.1708,2.6710) circle (0.1878mm);
  \fill (1.3376,2.7266) circle (0.1878mm);
  \fill (1.5000,2.7582) circle (0.1878mm);
  \fill (1.6624,2.7266) circle (0.1878mm);
  \fill (1.8292,2.6710) circle (0.1878mm);
  \fill (14.2657,1.3589) -- (14.1708,1.6710) -- (14.2538,1.4641) -- (14.3376,1.7266) -- (14.4179,1.4487) -- (14.5000,1.7582) -- (14.5821,1.4487) -- (14.6624,1.7266) -- (14.7462,1.4641) -- (14.8292,1.6710) -- (14.7343,1.3589) -- cycle;
  \fill (14.2341,1.2521) rectangle (14.7659,1.3512);
  \fill (14.1546,1.1546) -- (14.8454,1.1546) -- (14.7745,1.2452) -- (14.2255,1.2452) -- cycle;
  \fill (14.1708,1.6710) circle (0.1878mm);
  \fill (14.3376,1.7266) circle (0.1878mm);
  \fill (14.5000,1.7582) circle (0.1878mm);
  \fill (14.6624,1.7266) circle (0.1878mm);
  \fill (14.8292,1.6710) circle (0.1878mm);
  \fill (14.2657,0.3589) -- (14.1708,0.6710) -- (14.2538,0.4641) -- (14.3376,0.7266) -- (14.4179,0.4487) -- (14.5000,0.7582) -- (14.5821,0.4487) -- (14.6624,0.7266) -- (14.7462,0.4641) -- (14.8292,0.6710) -- (14.7343,0.3589) -- cycle;
  \fill (14.2341,0.2520) rectangle (14.7659,0.3512);
  \fill (14.1546,0.1546) -- (14.8454,0.1546) -- (14.7745,0.2452) -- (14.2255,0.2452) -- cycle;
  \fill (14.1708,0.6710) circle (0.1878mm);
  \fill (14.3376,0.7266) circle (0.1878mm);
  \fill (14.5000,0.7582) circle (0.1878mm);
  \fill (14.6624,0.7266) circle (0.1878mm);
  \fill (14.8292,0.6710) circle (0.1878mm);
  \node[above] at (1.0,15) {\small $L$};
  \node[above] at (14.0,15) {\small $t-L$};
  \node[right] at (15,2.0) {\small $t$};
  \node[below] at (7.5,-0.35) {\small $t=4$};
\end{scope}
\begin{scope}[xshift=96.620mm,x=2.780mm,y=2.780mm]
  \fill[black!11] (2,5) rectangle (13,16);
  \draw[line width=0.5pt] (2,5) rectangle (13,16);
  \node at (7.5,10.5) {\small $H[B]$};
  \draw[black!35,line width=0.15pt] (0,0) -- (0,5);
  \draw[black!35,line width=0.15pt] (1,0) -- (1,5);
  \draw[black!35,line width=0.15pt] (2,0) -- (2,5);
  \draw[black!35,line width=0.15pt] (3,0) -- (3,5);
  \draw[black!35,line width=0.15pt] (4,0) -- (4,5);
  \draw[black!35,line width=0.15pt] (5,0) -- (5,5);
  \draw[black!35,line width=0.15pt] (6,0) -- (6,5);
  \draw[black!35,line width=0.15pt] (7,0) -- (7,5);
  \draw[black!35,line width=0.15pt] (8,0) -- (8,5);
  \draw[black!35,line width=0.15pt] (9,0) -- (9,5);
  \draw[black!35,line width=0.15pt] (10,0) -- (10,5);
  \draw[black!35,line width=0.15pt] (11,0) -- (11,5);
  \draw[black!35,line width=0.15pt] (12,0) -- (12,5);
  \draw[black!35,line width=0.15pt] (13,0) -- (13,5);
  \draw[black!35,line width=0.15pt] (14,0) -- (14,5);
  \draw[black!35,line width=0.15pt] (15,0) -- (15,5);
  \draw[black!35,line width=0.15pt] (16,0) -- (16,5);
  \draw[black!35,line width=0.15pt] (0,0) -- (16,0);
  \draw[black!35,line width=0.15pt] (0,1) -- (16,1);
  \draw[black!35,line width=0.15pt] (0,2) -- (16,2);
  \draw[black!35,line width=0.15pt] (0,3) -- (16,3);
  \draw[black!35,line width=0.15pt] (0,4) -- (16,4);
  \draw[black!35,line width=0.15pt] (0,5) -- (16,5);
  \draw[black!35,line width=0.15pt] (0,5) -- (0,16);
  \draw[black!35,line width=0.15pt] (1,5) -- (1,16);
  \draw[black!35,line width=0.15pt] (2,5) -- (2,16);
  \draw[black!35,line width=0.15pt] (0,5) -- (2,5);
  \draw[black!35,line width=0.15pt] (0,6) -- (2,6);
  \draw[black!35,line width=0.15pt] (0,7) -- (2,7);
  \draw[black!35,line width=0.15pt] (0,8) -- (2,8);
  \draw[black!35,line width=0.15pt] (0,9) -- (2,9);
  \draw[black!35,line width=0.15pt] (0,10) -- (2,10);
  \draw[black!35,line width=0.15pt] (0,11) -- (2,11);
  \draw[black!35,line width=0.15pt] (0,12) -- (2,12);
  \draw[black!35,line width=0.15pt] (0,13) -- (2,13);
  \draw[black!35,line width=0.15pt] (0,14) -- (2,14);
  \draw[black!35,line width=0.15pt] (0,15) -- (2,15);
  \draw[black!35,line width=0.15pt] (0,16) -- (2,16);
  \draw[black!35,line width=0.15pt] (13,5) -- (13,16);
  \draw[black!35,line width=0.15pt] (14,5) -- (14,16);
  \draw[black!35,line width=0.15pt] (15,5) -- (15,16);
  \draw[black!35,line width=0.15pt] (16,5) -- (16,16);
  \draw[black!35,line width=0.15pt] (13,5) -- (16,5);
  \draw[black!35,line width=0.15pt] (13,6) -- (16,6);
  \draw[black!35,line width=0.15pt] (13,7) -- (16,7);
  \draw[black!35,line width=0.15pt] (13,8) -- (16,8);
  \draw[black!35,line width=0.15pt] (13,9) -- (16,9);
  \draw[black!35,line width=0.15pt] (13,10) -- (16,10);
  \draw[black!35,line width=0.15pt] (13,11) -- (16,11);
  \draw[black!35,line width=0.15pt] (13,12) -- (16,12);
  \draw[black!35,line width=0.15pt] (13,13) -- (16,13);
  \draw[black!35,line width=0.15pt] (13,14) -- (16,14);
  \draw[black!35,line width=0.15pt] (13,15) -- (16,15);
  \draw[black!35,line width=0.15pt] (13,16) -- (16,16);
  \draw[line width=0.5pt] (0,0) rectangle (16,16);
  \fill (13.2657,4.3589) -- (13.1708,4.6710) -- (13.2538,4.4641) -- (13.3376,4.7266) -- (13.4179,4.4487) -- (13.5000,4.7582) -- (13.5821,4.4487) -- (13.6624,4.7266) -- (13.7462,4.4641) -- (13.8292,4.6710) -- (13.7343,4.3589) -- cycle;
  \fill (13.2341,4.2520) rectangle (13.7659,4.3512);
  \fill (13.1546,4.1546) -- (13.8454,4.1546) -- (13.7745,4.2452) -- (13.2255,4.2452) -- cycle;
  \fill (13.1708,4.6710) circle (0.1878mm);
  \fill (13.3376,4.7266) circle (0.1878mm);
  \fill (13.5000,4.7582) circle (0.1878mm);
  \fill (13.6624,4.7266) circle (0.1878mm);
  \fill (13.8292,4.6710) circle (0.1878mm);
  \fill (14.2657,4.3589) -- (14.1708,4.6710) -- (14.2538,4.4641) -- (14.3376,4.7266) -- (14.4179,4.4487) -- (14.5000,4.7582) -- (14.5821,4.4487) -- (14.6624,4.7266) -- (14.7462,4.4641) -- (14.8292,4.6710) -- (14.7343,4.3589) -- cycle;
  \fill (14.2341,4.2520) rectangle (14.7659,4.3512);
  \fill (14.1546,4.1546) -- (14.8454,4.1546) -- (14.7745,4.2452) -- (14.2255,4.2452) -- cycle;
  \fill (14.1708,4.6710) circle (0.1878mm);
  \fill (14.3376,4.7266) circle (0.1878mm);
  \fill (14.5000,4.7582) circle (0.1878mm);
  \fill (14.6624,4.7266) circle (0.1878mm);
  \fill (14.8292,4.6710) circle (0.1878mm);
  \fill (1.2657,3.3589) -- (1.1708,3.6710) -- (1.2538,3.4641) -- (1.3376,3.7266) -- (1.4179,3.4487) -- (1.5000,3.7582) -- (1.5821,3.4487) -- (1.6624,3.7266) -- (1.7462,3.4641) -- (1.8292,3.6710) -- (1.7343,3.3589) -- cycle;
  \fill (1.2341,3.2521) rectangle (1.7659,3.3512);
  \fill (1.1546,3.1546) -- (1.8454,3.1546) -- (1.7745,3.2452) -- (1.2255,3.2452) -- cycle;
  \fill (1.1708,3.6710) circle (0.1878mm);
  \fill (1.3376,3.7266) circle (0.1878mm);
  \fill (1.5000,3.7582) circle (0.1878mm);
  \fill (1.6624,3.7266) circle (0.1878mm);
  \fill (1.8292,3.6710) circle (0.1878mm);
  \fill (0.2657,2.3589) -- (0.1708,2.6710) -- (0.2538,2.4641) -- (0.3376,2.7266) -- (0.4179,2.4487) -- (0.5000,2.7582) -- (0.5821,2.4487) -- (0.6624,2.7266) -- (0.7462,2.4641) -- (0.8292,2.6710) -- (0.7343,2.3589) -- cycle;
  \fill (0.2341,2.2521) rectangle (0.7659,2.3512);
  \fill (0.1546,2.1546) -- (0.8454,2.1546) -- (0.7745,2.2452) -- (0.2255,2.2452) -- cycle;
  \fill (0.1708,2.6710) circle (0.1878mm);
  \fill (0.3376,2.7266) circle (0.1878mm);
  \fill (0.5000,2.7582) circle (0.1878mm);
  \fill (0.6624,2.7266) circle (0.1878mm);
  \fill (0.8292,2.6710) circle (0.1878mm);
  \fill (15.2657,1.3589) -- (15.1708,1.6710) -- (15.2538,1.4641) -- (15.3376,1.7266) -- (15.4179,1.4487) -- (15.5000,1.7582) -- (15.5821,1.4487) -- (15.6624,1.7266) -- (15.7462,1.4641) -- (15.8292,1.6710) -- (15.7343,1.3589) -- cycle;
  \fill (15.2341,1.2521) rectangle (15.7659,1.3512);
  \fill (15.1546,1.1546) -- (15.8454,1.1546) -- (15.7745,1.2452) -- (15.2255,1.2452) -- cycle;
  \fill (15.1708,1.6710) circle (0.1878mm);
  \fill (15.3376,1.7266) circle (0.1878mm);
  \fill (15.5000,1.7582) circle (0.1878mm);
  \fill (15.6624,1.7266) circle (0.1878mm);
  \fill (15.8292,1.6710) circle (0.1878mm);
  \fill (15.2657,0.3589) -- (15.1708,0.6710) -- (15.2538,0.4641) -- (15.3376,0.7266) -- (15.4179,0.4487) -- (15.5000,0.7582) -- (15.5821,0.4487) -- (15.6624,0.7266) -- (15.7462,0.4641) -- (15.8292,0.6710) -- (15.7343,0.3589) -- cycle;
  \fill (15.2341,0.2520) rectangle (15.7659,0.3512);
  \fill (15.1546,0.1546) -- (15.8454,0.1546) -- (15.7745,0.2452) -- (15.2255,0.2452) -- cycle;
  \fill (15.1708,0.6710) circle (0.1878mm);
  \fill (15.3376,0.7266) circle (0.1878mm);
  \fill (15.5000,0.7582) circle (0.1878mm);
  \fill (15.6624,0.7266) circle (0.1878mm);
  \fill (15.8292,0.6710) circle (0.1878mm);
  \node[above] at (1.0,16) {\small $L$};
  \node[above] at (14.5,16) {\small $t-L$};
  \node[right] at (16,2.5) {\small $t$};
  \node[below] at (8.0,-0.35) {\small $t=5$};
\end{scope}
\end{tikzpicture}%
}

\newcommand{\figsfifteen}{%
\begin{tikzpicture}
\begin{scope}[xshift=0.0000mm,yshift=0.0000mm,x=4.0000mm,y=4.0000mm]
  \fill[black!12] (0,14) rectangle (1,15) (14,14) rectangle (15,15) (1,13) rectangle (2,14) (13,13) rectangle (14,14) (2,12) rectangle (3,13) (12,12) rectangle (13,13) (3,11) rectangle (4,12) (11,11) rectangle (12,12) (4,10) rectangle (5,11) (10,10) rectangle (11,11) (5,9) rectangle (6,10) (9,9) rectangle (10,10) (6,8) rectangle (7,9) (8,8) rectangle (9,9) (7,7) rectangle (8,8) (6,6) rectangle (7,7) (8,6) rectangle (9,7) (5,5) rectangle (6,6) (9,5) rectangle (10,6) (4,4) rectangle (5,5) (10,4) rectangle (11,5) (3,3) rectangle (4,4) (11,3) rectangle (12,4) (2,2) rectangle (3,3) (12,2) rectangle (13,3) (1,1) rectangle (2,2) (13,1) rectangle (14,2) (0,0) rectangle (1,1) (14,0) rectangle (15,1);
  \draw[black!30,line width=0.15pt] (0,0) -- (0,15);
  \draw[black!30,line width=0.15pt] (1,0) -- (1,15);
  \draw[black!30,line width=0.15pt] (2,0) -- (2,15);
  \draw[black!30,line width=0.15pt] (3,0) -- (3,15);
  \draw[black!30,line width=0.15pt] (4,0) -- (4,15);
  \draw[black!30,line width=0.15pt] (5,0) -- (5,15);
  \draw[black!30,line width=0.15pt] (6,0) -- (6,15);
  \draw[black!30,line width=0.15pt] (7,0) -- (7,15);
  \draw[black!30,line width=0.15pt] (8,0) -- (8,15);
  \draw[black!30,line width=0.15pt] (9,0) -- (9,15);
  \draw[black!30,line width=0.15pt] (10,0) -- (10,15);
  \draw[black!30,line width=0.15pt] (11,0) -- (11,15);
  \draw[black!30,line width=0.15pt] (12,0) -- (12,15);
  \draw[black!30,line width=0.15pt] (13,0) -- (13,15);
  \draw[black!30,line width=0.15pt] (14,0) -- (14,15);
  \draw[black!30,line width=0.15pt] (15,0) -- (15,15);
  \draw[black!30,line width=0.15pt] (0,0) -- (15,0);
  \draw[black!30,line width=0.15pt] (0,1) -- (15,1);
  \draw[black!30,line width=0.15pt] (0,2) -- (15,2);
  \draw[black!30,line width=0.15pt] (0,3) -- (15,3);
  \draw[black!30,line width=0.15pt] (0,4) -- (15,4);
  \draw[black!30,line width=0.15pt] (0,5) -- (15,5);
  \draw[black!30,line width=0.15pt] (0,6) -- (15,6);
  \draw[black!30,line width=0.15pt] (0,7) -- (15,7);
  \draw[black!30,line width=0.15pt] (0,8) -- (15,8);
  \draw[black!30,line width=0.15pt] (0,9) -- (15,9);
  \draw[black!30,line width=0.15pt] (0,10) -- (15,10);
  \draw[black!30,line width=0.15pt] (0,11) -- (15,11);
  \draw[black!30,line width=0.15pt] (0,12) -- (15,12);
  \draw[black!30,line width=0.15pt] (0,13) -- (15,13);
  \draw[black!30,line width=0.15pt] (0,14) -- (15,14);
  \draw[black!30,line width=0.15pt] (0,15) -- (15,15);
  \draw[line width=0.5pt] (0,0) rectangle (15,15);
  \fill (3.2657,14.3589) -- (3.1708,14.6710) -- (3.2538,14.4641) -- (3.3375,14.7266) -- (3.4179,14.4487) -- (3.5000,14.7582) -- (3.5821,14.4487) -- (3.6625,14.7266) -- (3.7462,14.4641) -- (3.8292,14.6710) -- (3.7343,14.3589) -- cycle;
  \fill (3.2341,14.2521) rectangle (3.7659,14.3512);
  \fill (3.1546,14.1546) -- (3.8454,14.1546) -- (3.7745,14.2452) -- (3.2255,14.2452) -- cycle;
  \fill (3.1708,14.6710) circle (0.2702mm);
  \fill (3.3375,14.7266) circle (0.2702mm);
  \fill (3.5000,14.7582) circle (0.2702mm);
  \fill (3.6625,14.7266) circle (0.2702mm);
  \fill (3.8292,14.6710) circle (0.2702mm);
  \fill (8.2657,14.3589) -- (8.1708,14.6710) -- (8.2538,14.4641) -- (8.3376,14.7266) -- (8.4179,14.4487) -- (8.5000,14.7582) -- (8.5821,14.4487) -- (8.6624,14.7266) -- (8.7462,14.4641) -- (8.8292,14.6710) -- (8.7343,14.3589) -- cycle;
  \fill (8.2341,14.2521) rectangle (8.7659,14.3512);
  \fill (8.1546,14.1546) -- (8.8454,14.1546) -- (8.7745,14.2452) -- (8.2255,14.2452) -- cycle;
  \fill (8.1708,14.6710) circle (0.2702mm);
  \fill (8.3376,14.7266) circle (0.2702mm);
  \fill (8.5000,14.7582) circle (0.2702mm);
  \fill (8.6624,14.7266) circle (0.2702mm);
  \fill (8.8292,14.6710) circle (0.2702mm);
  \fill (5.2657,13.3589) -- (5.1708,13.6710) -- (5.2538,13.4641) -- (5.3376,13.7266) -- (5.4179,13.4487) -- (5.5000,13.7582) -- (5.5821,13.4487) -- (5.6624,13.7266) -- (5.7462,13.4641) -- (5.8292,13.6710) -- (5.7343,13.3589) -- cycle;
  \fill (5.2341,13.2521) rectangle (5.7659,13.3512);
  \fill (5.1546,13.1546) -- (5.8454,13.1546) -- (5.7745,13.2452) -- (5.2255,13.2452) -- cycle;
  \fill (5.1708,13.6710) circle (0.2702mm);
  \fill (5.3376,13.7266) circle (0.2702mm);
  \fill (5.5000,13.7582) circle (0.2702mm);
  \fill (5.6624,13.7266) circle (0.2702mm);
  \fill (5.8292,13.6710) circle (0.2702mm);
  \fill (14.2657,12.3589) -- (14.1708,12.6710) -- (14.2538,12.4641) -- (14.3376,12.7266) -- (14.4179,12.4487) -- (14.5000,12.7582) -- (14.5821,12.4487) -- (14.6624,12.7266) -- (14.7462,12.4641) -- (14.8292,12.6710) -- (14.7343,12.3589) -- cycle;
  \fill (14.2341,12.2521) rectangle (14.7659,12.3512);
  \fill (14.1546,12.1546) -- (14.8454,12.1546) -- (14.7745,12.2452) -- (14.2255,12.2452) -- cycle;
  \fill (14.1708,12.6710) circle (0.2702mm);
  \fill (14.3376,12.7266) circle (0.2702mm);
  \fill (14.5000,12.7582) circle (0.2702mm);
  \fill (14.6624,12.7266) circle (0.2702mm);
  \fill (14.8292,12.6710) circle (0.2702mm);
  \fill (2.2657,11.3589) -- (2.1708,11.6710) -- (2.2538,11.4641) -- (2.3375,11.7266) -- (2.4179,11.4487) -- (2.5000,11.7582) -- (2.5821,11.4487) -- (2.6625,11.7266) -- (2.7462,11.4641) -- (2.8292,11.6710) -- (2.7343,11.3589) -- cycle;
  \fill (2.2341,11.2521) rectangle (2.7659,11.3512);
  \fill (2.1546,11.1546) -- (2.8454,11.1546) -- (2.7745,11.2452) -- (2.2255,11.2452) -- cycle;
  \fill (2.1708,11.6710) circle (0.2702mm);
  \fill (2.3375,11.7266) circle (0.2702mm);
  \fill (2.5000,11.7582) circle (0.2702mm);
  \fill (2.6625,11.7266) circle (0.2702mm);
  \fill (2.8292,11.6710) circle (0.2702mm);
  \fill (6.2657,10.3589) -- (6.1708,10.6710) -- (6.2538,10.4641) -- (6.3376,10.7266) -- (6.4179,10.4487) -- (6.5000,10.7582) -- (6.5821,10.4487) -- (6.6624,10.7266) -- (6.7462,10.4641) -- (6.8292,10.6710) -- (6.7343,10.3589) -- cycle;
  \fill (6.2341,10.2521) rectangle (6.7659,10.3512);
  \fill (6.1546,10.1546) -- (6.8454,10.1546) -- (6.7745,10.2452) -- (6.2255,10.2452) -- cycle;
  \fill (6.1708,10.6710) circle (0.2702mm);
  \fill (6.3376,10.7266) circle (0.2702mm);
  \fill (6.5000,10.7582) circle (0.2702mm);
  \fill (6.6624,10.7266) circle (0.2702mm);
  \fill (6.8292,10.6710) circle (0.2702mm);
  \fill (11.2657,10.3589) -- (11.1708,10.6710) -- (11.2538,10.4641) -- (11.3376,10.7266) -- (11.4179,10.4487) -- (11.5000,10.7582) -- (11.5821,10.4487) -- (11.6624,10.7266) -- (11.7462,10.4641) -- (11.8292,10.6710) -- (11.7343,10.3589) -- cycle;
  \fill (11.2341,10.2521) rectangle (11.7659,10.3512);
  \fill (11.1546,10.1546) -- (11.8454,10.1546) -- (11.7745,10.2452) -- (11.2255,10.2452) -- cycle;
  \fill (11.1708,10.6710) circle (0.2702mm);
  \fill (11.3376,10.7266) circle (0.2702mm);
  \fill (11.5000,10.7582) circle (0.2702mm);
  \fill (11.6624,10.7266) circle (0.2702mm);
  \fill (11.8292,10.6710) circle (0.2702mm);
  \fill (14.2657,9.3589) -- (14.1708,9.6710) -- (14.2538,9.4641) -- (14.3376,9.7266) -- (14.4179,9.4487) -- (14.5000,9.7582) -- (14.5821,9.4487) -- (14.6624,9.7266) -- (14.7462,9.4641) -- (14.8292,9.6710) -- (14.7343,9.3589) -- cycle;
  \fill (14.2341,9.2521) rectangle (14.7659,9.3512);
  \fill (14.1546,9.1546) -- (14.8454,9.1546) -- (14.7745,9.2452) -- (14.2255,9.2452) -- cycle;
  \fill (14.1708,9.6710) circle (0.2702mm);
  \fill (14.3376,9.7266) circle (0.2702mm);
  \fill (14.5000,9.7582) circle (0.2702mm);
  \fill (14.6624,9.7266) circle (0.2702mm);
  \fill (14.8292,9.6710) circle (0.2702mm);
  \fill (1.2657,8.3589) -- (1.1708,8.6710) -- (1.2538,8.4641) -- (1.3376,8.7266) -- (1.4179,8.4487) -- (1.5000,8.7582) -- (1.5821,8.4487) -- (1.6624,8.7266) -- (1.7462,8.4641) -- (1.8292,8.6710) -- (1.7343,8.3589) -- cycle;
  \fill (1.2341,8.2521) rectangle (1.7659,8.3512);
  \fill (1.1546,8.1546) -- (1.8454,8.1546) -- (1.7745,8.2452) -- (1.2255,8.2452) -- cycle;
  \fill (1.1708,8.6710) circle (0.2702mm);
  \fill (1.3376,8.7266) circle (0.2702mm);
  \fill (1.5000,8.7582) circle (0.2702mm);
  \fill (1.6624,8.7266) circle (0.2702mm);
  \fill (1.8292,8.6710) circle (0.2702mm);
  \fill (7.2657,8.3589) -- (7.1708,8.6710) -- (7.2538,8.4641) -- (7.3376,8.7266) -- (7.4179,8.4487) -- (7.5000,8.7582) -- (7.5821,8.4487) -- (7.6624,8.7266) -- (7.7462,8.4641) -- (7.8292,8.6710) -- (7.7343,8.3589) -- cycle;
  \fill (7.2341,8.2521) rectangle (7.7659,8.3512);
  \fill (7.1546,8.1546) -- (7.8454,8.1546) -- (7.7745,8.2452) -- (7.2255,8.2452) -- cycle;
  \fill (7.1708,8.6710) circle (0.2702mm);
  \fill (7.3376,8.7266) circle (0.2702mm);
  \fill (7.5000,8.7582) circle (0.2702mm);
  \fill (7.6624,8.7266) circle (0.2702mm);
  \fill (7.8292,8.6710) circle (0.2702mm);
  \fill (13.2657,7.3589) -- (13.1708,7.6710) -- (13.2538,7.4641) -- (13.3376,7.7266) -- (13.4179,7.4487) -- (13.5000,7.7582) -- (13.5821,7.4487) -- (13.6624,7.7266) -- (13.7462,7.4641) -- (13.8292,7.6710) -- (13.7343,7.3589) -- cycle;
  \fill (13.2341,7.2520) rectangle (13.7659,7.3512);
  \fill (13.1546,7.1546) -- (13.8454,7.1546) -- (13.7745,7.2452) -- (13.2255,7.2452) -- cycle;
  \fill (13.1708,7.6710) circle (0.2702mm);
  \fill (13.3376,7.7266) circle (0.2702mm);
  \fill (13.5000,7.7582) circle (0.2702mm);
  \fill (13.6624,7.7266) circle (0.2702mm);
  \fill (13.8292,7.6710) circle (0.2702mm);
  \fill (13.2657,6.3589) -- (13.1708,6.6710) -- (13.2538,6.4641) -- (13.3376,6.7266) -- (13.4179,6.4487) -- (13.5000,6.7582) -- (13.5821,6.4487) -- (13.6624,6.7266) -- (13.7462,6.4641) -- (13.8292,6.6710) -- (13.7343,6.3589) -- cycle;
  \fill (13.2341,6.2520) rectangle (13.7659,6.3512);
  \fill (13.1546,6.1546) -- (13.8454,6.1546) -- (13.7745,6.2452) -- (13.2255,6.2452) -- cycle;
  \fill (13.1708,6.6710) circle (0.2702mm);
  \fill (13.3376,6.7266) circle (0.2702mm);
  \fill (13.5000,6.7582) circle (0.2702mm);
  \fill (13.6624,6.7266) circle (0.2702mm);
  \fill (13.8292,6.6710) circle (0.2702mm);
  \fill (0.2657,5.3589) -- (0.1708,5.6710) -- (0.2538,5.4641) -- (0.3376,5.7266) -- (0.4179,5.4487) -- (0.5000,5.7582) -- (0.5821,5.4487) -- (0.6624,5.7266) -- (0.7462,5.4641) -- (0.8292,5.6710) -- (0.7343,5.3589) -- cycle;
  \fill (0.2341,5.2520) rectangle (0.7659,5.3512);
  \fill (0.1546,5.1546) -- (0.8454,5.1546) -- (0.7745,5.2452) -- (0.2255,5.2452) -- cycle;
  \fill (0.1708,5.6710) circle (0.2702mm);
  \fill (0.3376,5.7266) circle (0.2702mm);
  \fill (0.5000,5.7582) circle (0.2702mm);
  \fill (0.6624,5.7266) circle (0.2702mm);
  \fill (0.8292,5.6710) circle (0.2702mm);
  \fill (2.2657,4.3589) -- (2.1708,4.6710) -- (2.2538,4.4641) -- (2.3375,4.7266) -- (2.4179,4.4487) -- (2.5000,4.7582) -- (2.5821,4.4487) -- (2.6625,4.7266) -- (2.7462,4.4641) -- (2.8292,4.6710) -- (2.7343,4.3589) -- cycle;
  \fill (2.2341,4.2520) rectangle (2.7659,4.3512);
  \fill (2.1546,4.1546) -- (2.8454,4.1546) -- (2.7745,4.2452) -- (2.2255,4.2452) -- cycle;
  \fill (2.1708,4.6710) circle (0.2702mm);
  \fill (2.3375,4.7266) circle (0.2702mm);
  \fill (2.5000,4.7582) circle (0.2702mm);
  \fill (2.6625,4.7266) circle (0.2702mm);
  \fill (2.8292,4.6710) circle (0.2702mm);
  \fill (0.2657,3.3589) -- (0.1708,3.6710) -- (0.2538,3.4641) -- (0.3376,3.7266) -- (0.4179,3.4487) -- (0.5000,3.7582) -- (0.5821,3.4487) -- (0.6624,3.7266) -- (0.7462,3.4641) -- (0.8292,3.6710) -- (0.7343,3.3589) -- cycle;
  \fill (0.2341,3.2521) rectangle (0.7659,3.3512);
  \fill (0.1546,3.1546) -- (0.8454,3.1546) -- (0.7745,3.2452) -- (0.2255,3.2452) -- cycle;
  \fill (0.1708,3.6710) circle (0.2702mm);
  \fill (0.3376,3.7266) circle (0.2702mm);
  \fill (0.5000,3.7582) circle (0.2702mm);
  \fill (0.6624,3.7266) circle (0.2702mm);
  \fill (0.8292,3.6710) circle (0.2702mm);
  \fill (5.2657,2.3589) -- (5.1708,2.6710) -- (5.2538,2.4641) -- (5.3376,2.7266) -- (5.4179,2.4487) -- (5.5000,2.7582) -- (5.5821,2.4487) -- (5.6624,2.7266) -- (5.7462,2.4641) -- (5.8292,2.6710) -- (5.7343,2.3589) -- cycle;
  \fill (5.2341,2.2521) rectangle (5.7659,2.3512);
  \fill (5.1546,2.1546) -- (5.8454,2.1546) -- (5.7745,2.2452) -- (5.2255,2.2452) -- cycle;
  \fill (5.1708,2.6710) circle (0.2702mm);
  \fill (5.3376,2.7266) circle (0.2702mm);
  \fill (5.5000,2.7582) circle (0.2702mm);
  \fill (5.6624,2.7266) circle (0.2702mm);
  \fill (5.8292,2.6710) circle (0.2702mm);
  \fill (9.2657,1.3589) -- (9.1708,1.6710) -- (9.2538,1.4641) -- (9.3376,1.7266) -- (9.4179,1.4487) -- (9.5000,1.7582) -- (9.5821,1.4487) -- (9.6624,1.7266) -- (9.7462,1.4641) -- (9.8292,1.6710) -- (9.7343,1.3589) -- cycle;
  \fill (9.2341,1.2521) rectangle (9.7659,1.3512);
  \fill (9.1546,1.1546) -- (9.8454,1.1546) -- (9.7745,1.2452) -- (9.2255,1.2452) -- cycle;
  \fill (9.1708,1.6710) circle (0.2702mm);
  \fill (9.3376,1.7266) circle (0.2702mm);
  \fill (9.5000,1.7582) circle (0.2702mm);
  \fill (9.6624,1.7266) circle (0.2702mm);
  \fill (9.8292,1.6710) circle (0.2702mm);
  \fill (10.2657,1.3589) -- (10.1708,1.6710) -- (10.2538,1.4641) -- (10.3376,1.7266) -- (10.4179,1.4487) -- (10.5000,1.7582) -- (10.5821,1.4487) -- (10.6624,1.7266) -- (10.7462,1.4641) -- (10.8292,1.6710) -- (10.7343,1.3589) -- cycle;
  \fill (10.2341,1.2521) rectangle (10.7659,1.3512);
  \fill (10.1546,1.1546) -- (10.8454,1.1546) -- (10.7745,1.2452) -- (10.2255,1.2452) -- cycle;
  \fill (10.1708,1.6710) circle (0.2702mm);
  \fill (10.3376,1.7266) circle (0.2702mm);
  \fill (10.5000,1.7582) circle (0.2702mm);
  \fill (10.6624,1.7266) circle (0.2702mm);
  \fill (10.8292,1.6710) circle (0.2702mm);
  \fill (4.2657,0.3589) -- (4.1708,0.6710) -- (4.2538,0.4641) -- (4.3376,0.7266) -- (4.4179,0.4487) -- (4.5000,0.7582) -- (4.5821,0.4487) -- (4.6624,0.7266) -- (4.7462,0.4641) -- (4.8292,0.6710) -- (4.7343,0.3589) -- cycle;
  \fill (4.2341,0.2520) rectangle (4.7659,0.3512);
  \fill (4.1546,0.1546) -- (4.8454,0.1546) -- (4.7745,0.2452) -- (4.2255,0.2452) -- cycle;
  \fill (4.1708,0.6710) circle (0.2702mm);
  \fill (4.3376,0.7266) circle (0.2702mm);
  \fill (4.5000,0.7582) circle (0.2702mm);
  \fill (4.6624,0.7266) circle (0.2702mm);
  \fill (4.8292,0.6710) circle (0.2702mm);
  \fill (12.2657,0.3589) -- (12.1708,0.6710) -- (12.2538,0.4641) -- (12.3376,0.7266) -- (12.4179,0.4487) -- (12.5000,0.7582) -- (12.5821,0.4487) -- (12.6624,0.7266) -- (12.7462,0.4641) -- (12.8292,0.6710) -- (12.7343,0.3589) -- cycle;
  \fill (12.2341,0.2520) rectangle (12.7659,0.3512);
  \fill (12.1546,0.1546) -- (12.8454,0.1546) -- (12.7745,0.2452) -- (12.2255,0.2452) -- cycle;
  \fill (12.1708,0.6710) circle (0.2702mm);
  \fill (12.3376,0.7266) circle (0.2702mm);
  \fill (12.5000,0.7582) circle (0.2702mm);
  \fill (12.6624,0.7266) circle (0.2702mm);
  \fill (12.8292,0.6710) circle (0.2702mm);
\end{scope}
\end{tikzpicture}%
}

\newcommand{\figcorners}{%
}

\begin{document}

\title{Closing the gap and settling the problem of queens\\ on an $n\times n$ board, each attacking at most one other}
\author{Kristina Ago$^1$, Bojan Ba\v si\'c$^{1,}$\thanks{Corresponding author: bojan.basic@dmi.uns.ac.rs}\ , Radojka Ciganovi\'c$^2$\\[0.5ex]
\em \small $^1$Department of Mathematics and Informatics, University of Novi Sad,\\
\em \small Trg Dositeja Obradovi\'ca 4, 21000 Novi Sad, Serbia\\
\small \tt kristina.ago@dmi.uns.ac.rs, bojan.basic@dmi.uns.ac.rs\\[1ex]
\em \small $^2$Faculty of Technical Sciences, University of Novi Sad,\\
\em \small Trg Dositeja Obradovi\'ca 6, 21000 Novi Sad, Serbia\\
\small \tt ciganovic.radojka@uns.ac.rs
}
\date{}

\maketitle

\begin{abstract}

Let $q(n)$ denote the largest number of queens that can be placed on an
$n\times n$ chessboard so that no queen attacks more than one other queen. We
prove that $q(n)=\lfloor4n/3\rfloor$ for every $n\geqslant6$, and that
$q(n)=n$ for $n\leqslant5$, which settles a previously conjectural value. As a corollary, we also settle that, in the version of the problem where each queen attacks \emph{exactly} one other queen, the answer is $2\lfloor2n/3\rfloor$, again as previously conjectured.


\end{abstract}

\section{Introduction}

Few recreational problems have proved as durable as that of placing eight
queens on a chessboard so that no two of them attack each other. It was posed
in $1848$ by the chess composer Max Bezzel~\cite{bezzel} (writing under the pseudonym \emph{Schachfreund}) in the
\emph{Berliner Schachzeitung,} and restated two years later by Nauck \cite{nauck} in the
\emph{Illustrirte Zeitung;} it was Nauck's note that caught the eye of Gauss,
who worked on the problem in his correspondence with Schumacher and at first
found only $72$ of the $92$ solutions. The frequently repeated claim that Gauss
posed the problem, or that he solved it completely, is a textbook illustration
of how a historical error propagates; the matter was disentangled by Campbell~\cite{campbell}. The general problem on an $n\times n$ board was posed by Lionnet~\cite{lionett} in 1869, and the first proof that $n$ nonattacking queens
can be placed whenever $n\geqslant4$ is due to Pauls~\cite{pauls} (and not, as is often written, to Ahrens, though his book \cite{ahrens} may have contributed significantly to popularizing the subject).

What has kept the problem alive is that nearly every question about it other
than the existence question turns out to be hard. As an example, we mention counting; it took more than twenty years to settle the problem, which happened only quite recently. Let $Q(n)$ denote the number of ways of placing $n$ nonattacking
queens on an $n\times n$ board. Rivin, Vardi, and Zimmermann \cite{rivin}
conjectured that $\log Q(n)$ grows like $n\log n$. Luria \cite{luria} proved an upper bound of the
expected shape; matching lower bounds were obtained by Bowtell and Keevash
\cite{bowtell} and, independently, by Luria and Simkin~\cite{luriasimkin};
and finally Simkin \cite{simkin} showed that $Q(n)=((1\pm o(1))ne^{-\alpha})^n$
for a constant $\alpha$, which he located to within $3\cdot10^{-3}$ and which
Nobel, Agrawal, and Boyd \cite{nobel} have since computed to nine decimal
places.

For various variations of this setting, the reader may check some of the surveys \cite{watkins,bellstevens,hedetniemi}. One more result that we want to point out in particular is a classical result of P\'olya \cite{polya} (published in an
appendix to the second volume of the already mentioned Ahrens's book), which 
states that $n$ non-attacking queens can be placed on a toroidal $n\times n$ board if and only if $\gcd(n,6)=1$. Although toroidal boards could seem like quite an abstract notion, not really occurring in everyday life, it turns out (surprisingly?) that they have a significant role in many other more natural-looking problems, and in particular, they are in the center of the present article.

The problem we are focusing on in the present article is the following one. What is the maximal number of queens that can be placed on an $n\times n$ board, in such a way that each of them attacks at most one other queen? The problem probably appeared in print for the first time in the book of Gik \cite{gik}. He proved the upper bound of $\lfloor\frac{4n}3\rfloor$ and showed that $10$ queens can be achieved on an $8\times 8$ board. The problem was brought up again by Kim \cite{kim}, and this reference was then popularized by Gardner \cite{gardner}, which is probably why the problem is often attributed to Kim (and another reason could be that, as Gardner states in his article, Kim thought of the problem already while in high school, in 1975, which is one year before Gik's book). There we can see that other variants of the problem can also be considered, namely, when ``at most one'' is replaced by ``at most $k$'', or when \emph{exactly} $k$ attacks are prescribed instead of at most $k$. Hayes \cite{hayes} settled some of the cases in 1992, and presented the best at-the-time upper bound on other cases.

Therefore, here we focus on the $k=1$ case, \emph{at most} one attack per queen. However, it turns out that ``at most one'' and ``exactly one'' versions are tightly connected, and once we settle the first one, the second one also follows quite cheaply (see Corollary \ref{exac}). We first present some further historical checkpoints. In August 2008, the IBM Corporation posed the problem on its website as a challenge: in particular, it was required to find the best possible arrangement on a $30\times 30$ board (and also an $8\times 8$ board, which is less relevant). The intended solution was to settle the challenge by some efficient computer-based techniques, in particular, by solving it as a Constraint Satisfaction Problem. The sequence of proved optimal values appeared in the OEIS in 2015 \cite{oeis}; that year, Resta \cite{resta} managed to produce optimal solutions for all $n\times n$ boards up to $n=30$, and based on that, he conjectured that the correct value is $\lfloor\frac{4n}3\rfloor$ for every $n$ greater than $5$. The case $n=2017$ was posed by the second-named author of the present article, in collaboration with other members of the problem selection committee, at the Serbian Mathematical Olympiad 2017~\cite{smo}. In this case, the answer again turned out to be $\lfloor\frac{4n}3\rfloor$, but what should also be mentioned here is that this marked possibly the first time when the idea of generating larger boards by embedding toroidal boards into smaller ``host'' boards was applied to this particular problem. The idea itself was certainly not new, as it can be traced all the way back to P\'olya, who used that approach in his already mentioned article, but for the original, non-attacking queens problem, not for a limited number of attacks. Also, the solution page of the mentioned IBM challenge gives a construction for $n=30$ sent by one of the solvers, which is exactly the one that can be obtained by these toroidal embeddings; however, the necessity of considering tori was never understood on that page, and in fact, the claim there that this pattern can be generalized to any board whose side length is divisible by $6$ is actually erroneous.

A real milestone was reached by Makay \cite{makay}, who managed to unleash the full power of the toroidal-embedding machinery and prove that the optimal value is $\lfloor\frac{4n}3\rfloor$ for all large enough $n$. In particular, his construction works for all numbers from $354$ onward, as well as for many values below this cutoff, but significantly many ``holes'' remain. The purpose of this article is to develop a few tricks that are, taken together, strong enough to cover all the remaining cases and settle the problem in full. 

The initial plan was to use Resta's list \cite{resta} up to $n=30$ as known, and to make the article self-contained for all the other values. However, while we were in the finishing stage of preparation of the article, an extended list by Healy \cite{healy} appeared, covering all the cases up to $n=72$. While the cases between $n=31$ and $n=72$ could in principle be handled by our techniques (or some slight modifications that are not in this final version of the article), a few exceptional cases would remain and they would need to be taken care of individually. For that reason, we decided to use Healy's table for $n$ up to $72$ and to cover only the rest, which spared us and the reader such inconveniences, and additionally, enabled us to make some parts simpler.

\emph{Note added after completion of the manuscript.}
As fate would have it, after this manuscript was finished and when the authors were ready to submit it for publication, an independent manuscript by Healy \cite{healyproof} appeared (four days before the submission date of the present manuscript), solving the same problem. What we can say about the comparison between his manuscript and ours is that (of course) the essence of both is again toroidal embedding. What he uses in addition is polishing the idea of adding a new board at the bottom right as much as possible; he generated such boards (with some special properties) up to size $29$, and by combining them, he can solve the problem on all the boards from $n=216$ onward; the same argument also covers a selection of numbers below $216$, while for other ones he relies on some exceptional constructions. An attractive feature of his approach is therefore that its uniform large-order part is built around essentially one additional mechanism beyond the toroidal construction. Our proof uses several different extension mechanisms instead. On the other hand, it is substantially less dependent on a catalogue of computer-generated auxiliary boards: most of our auxiliary boards have a simple visible structure, and furthermore, at least for some value of $n$ onward (after the need for some ``more exotic'' patches stops), the construction becomes recursive and requires no further search or stored configurations, and probably can be easily reconstructed by a human.

\section{Preliminaries and the upper bound}
\label{s2}

We work on the $n\times n$ board, whose rows and columns are numbered
$1,2,\dots,n$; a cell is written as a pair $(r,c)$, its row first. In our figures, rows are enumerated from top to bottom, and columns from left to right. Although this is different from the usual chess convention (where rows are enumerated from bottom to top), we feel that the chess analogy becomes quite muddled once we reach boards of much larger size than $8\times 8$, and thus opting for usual convention for matrices instead of chess seemed to us like a more natural choice. With every
cell we associate four numbers: its row $r$, its column $c$, its
\emph{difference} $r-c$, and its \emph{sum} $r+c$. Two cells lie on a common
line of the board (on a row, a column, a diagonal, or an antidiagonal)
precisely when they agree in one of these four numbers, and two different cells
can agree in at most one of them. (Note: here, by \emph{antidiagonal,} we mean the set of cells with a fixed sum. This usage will be kept in the article when we find necessary to make a distinction between the fixed difference and the fixed sum, but in some other places, when the context is clear, both those kinds of lines will be simply referred to as diagonals. We believe that this will not cause any confusion.)

A \emph{configuration} is a set of cells, thought of as the cells occupied by
queens. The queens are ordinary chess queens, so two of them \emph{attack} each
other if they lie on a common line and no queen stands between them. We call a
configuration \emph{admissible} if no queen attacks more than one other queen,
we let $q(n)$ be the largest number of queens in an admissible configuration on
the $n\times n$ board, and we call an admissible configuration with $q(n)$
queens \emph{optimal.}

\begin{theorem}
\label{upper}
For every positive integer $n$ we have
\[
q(n)\leqslant\left\lfloor\frac{4n}{3}\right\rfloor.
\]
\end{theorem}

\begin{proof}
Let $m$ queens be placed on an $n\times n$ board for some $m$, and suppose $m>n$. No row contains
more than two queens, so there are at least $m - n$ rows with two queens, and hence there are
at most $m - 2(m - n)$, which is $ 2n - m $ queens that are alone in their rows. Similarly, at most $2n - m$
queens are alone in their columns. On the other hand, every queen is alone in its
row or column, which means that every queen is accounted for in the previous count, that is,
$  2(2n - m)\geqslant m$. Thus $m \leqslant \frac{ 4n}
3 $. As $m$ is an integer, the  conclusion follows.\end{proof}

For $n\leqslant5$ the bound of Theorem \ref{upper} is not attained: an
inspection of the finitely many configurations gives $q(n)=n$ for
$1\leqslant n\leqslant5$, whereas $\lfloor4n/3\rfloor$ equals $4$, $5$, and $6$
for $n=3,4,5$.

From $n=6$ on the bound is attained, and for small boards this is known
explicitly. Healy \cite{healy} lists an optimal configuration for every
$n$ up to $72$, and an earlier list of Resta \cite{resta} does the same for
$6\leqslant n\leqslant30$; in particular $q(n)=\lfloor4n/3\rfloor$ for
$6\leqslant n\leqslant72$. These two lists are the only external input of this
article (and, actually, only some smaller parts of them are really needed, but we decided that it is cleaner to reference them instead of presenting separate constructions for the exceptional values we need). The result is
the following.

\begin{theorem}
\label{main}
For every positive integer $n$ we have
\[
q(n)=
\left\{
\begin{array}{ll}
n,&\text{if $1\leqslant n\leqslant5$;}\\[1mm]
\left\lfloor\frac{4n}{3}\right\rfloor,&\text{if $n\geqslant6$.}
\end{array}\right.
\]
\end{theorem}

In view of Theorem \ref{upper} and of the two lists just referenced, what remains
is to construct, for every $n\geqslant73$, an admissible configuration with
$\lfloor4n/3\rfloor$ queens. This is done in Section \ref{s5}.

 Theorem \ref{upper} has an easy corollary that settles the problem with queens attacking exactly one other. Namely, note that, with this requirement, in any admissible configuration queens can be split into pairs where queens in the same pair attack each other. Therefore, the number of queens in that case cannot be odd. And then we see that, for the optimal placement of queens in this variation, we can keep the same optimal configuration for our main problem whenever the number of queens is even, while when it is odd, we remove the single queen that does not attack any other queen. Therefore, we have the following corollary.

 \begin{corollary}\label{exac}
     The maximal number of queens that can be placed on an $n\times n$ board in such a way that each of them attacks exactly on other queen equals $0,2,2,4,4$ for $n=1,2,3,4,5$, respectively, and equals $2\left\lfloor\frac{2n}{3}\right\rfloor$ for $n\geqslant 6$.
 \end{corollary}

\section{Auxiliary constructions}\label{s3}

We begin with four small configurations, called the \emph{hosts.} In addition to being optimal, each of them leaves a strip of empty diagonals around the main diagonal. For a square in row $r$ and column $c$, its distance from the main diagonal is $|r-c|$. The \emph{clearance} of a host is the smallest distance attained by one of its queens.

The four side lengths we use for the hosts are $6$, $12$, $15$, and $18$. The four hosts are shown in Figure \ref{hosts}. Their clearances are, respectively, $1$, $2$, $2$, and $2$. Two of them are taken from \cite{resta}, namely, the $6\times6$ configuration  and the $12\times12$ configuration (the first one is flipped here). The $15\times15$ and $18\times18$ configurations displayed in
\cite{healy} and in \cite{resta} have queens next to the main diagonal and are
therefore of no use to us. That is why we devised  new ones ourselves. (Remark: it might catch an eye that $H_{18}$ seems quite more regular than $H_{15}$, but there is a simple explanation for this. Both of them were, of course, found by computer. We searched for $H_{15}$ without any restrictions on the search space, and this was the first one that appeared. Then we decided to make the algorithm faster by restricting it to look only for centrally symmetric configurations, hoping that such one existed, and indeed, this is how $H_{18}$ was found.) That all four are admissible
and have $8$, $16$, $20$, and $24$ queens is checked directly, and so is the following further piece of information, which
will be needed in Subsection \ref{ss33}: the main antidiagonal of the board is empty in $H_6$, in $H_{12}$, and in
$H_{18}$, but not in $H_{15}$.

\begin{figure}[htb]
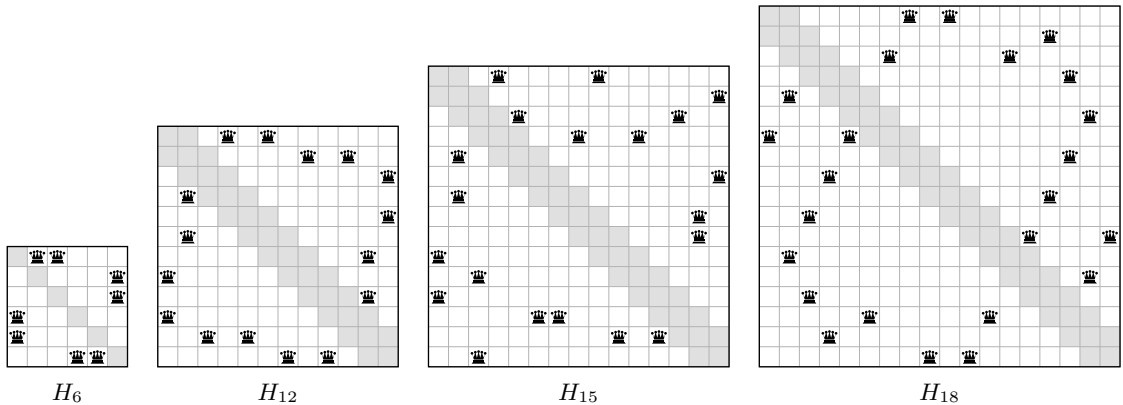

\begin{center}
\fighosts
\end{center}
\caption{The four hosts. The shaded cells are the $2d-1$ central diagonals,
which a host must avoid.}
\label{hosts}
\end{figure}

\subsection{Embedment}
\label{ss31}

This is the main construction. We explain it in plain words. Let $B$ be a positive integer with $\gcd(B,6)=1$. On a $B\times B$ board, put one queen in each row, namely, in row $i$ put the queen in column $2i$, with the column number read modulo $B$.

We claim that such queens do not attack each other, and that, moreover, the same is true even if the $B\times B$ board considered is a torus. Though this is well-known, we just sketch a proof. Namely, the column numbers run through all residues modulo $B$, as $2i\equiv 2j\pmod{B}$ implies $i\equiv j\pmod B$ (since $2\nmid B$), and thus $i=j$. The same is true for both diagonal directions. Namely, the queens $(i,2i)$ and $(j,2j)$ attack along one diagonal if and only if $i-2i\equiv j-2j\pmod{B}$, which reduces to $i\equiv j\pmod B$, that is, $i=j$. Finally, they attack along the other diagonal if and only if $i+2i\equiv j+2j\pmod B$, and this again reduces to $i\equiv j\pmod B$ and finally $i=j$, where it was used that $3\nmid B$.

We now replace each queen on some initial board by this toroidal pattern (the initial board will usually be one of the hosts, but the construction described here is general). Let $H$ be an arrangement of queens on an $A\times A$ board. Divide an $AB\times AB$ board into $A^2$ blocks of size $B\times B$. Take an occupied square of $H$, and let $R$ and $C$ be its row and column. Put a translated copy of the toroidal pattern into the corresponding block. Leave the blocks belonging to empty squares empty. We denote the resulting configuration by $H[B]$. 

\begin{theorem}
\label{substthm}
Let $H$ be an admissible configuration on an $A\times A$ board, and let $B$ be any positive integer with $\gcd(B,6)=1$. Then the configuration $H[B]$ is an admissible  configuration on the $AB\times AB$ board.
\end{theorem}

\begin{proof}
For $1\leqslant R,C\leqslant A$, let the block
$(R,C)$ of the $AB\times AB$ board consist of the cells
$(B(R-1)+i,B(C-1)+j)$ with $1\leqslant i,j\leqslant B$, and let $p$ be the
permutation of $\{1,2,\dots,B\}$ with $p(i)\equiv2i\pmod B$. The queen in the row $i$ in the block $(R,C)$ occupies the cell
$(B(R-1)+i,B(C-1)+p(i))$ in $H[B]$; therefore, its difference and its sum are
\[
B(R-C)+i-p(i)\qquad\text{and}\qquad B(R+C-2)+i+p(i).
\]
Note that if two queens on $H[B]$ have their assigned $i_1$ and $i_2$ distinct (also allowing $R$ and $C$ to vary), they cannot attack each other, as that would imply something of $i_1\equiv i_2\pmod B$, $p(i_1)\equiv p(i_2)\pmod B$, $i_1-p(i_1)\equiv i_2-p(i_2)\pmod B$, or $i_1+p(i_1)\equiv i_2+p(i_2)\pmod B$, all of which would mean that the queens in rows $i_1$ and $i_2$ attack each other on the underlying $B\times B$ torus, which is impossible. Therefore, if any two queens on $H[B]$ attack each other, they have to be placed in exactly the same positions within their corresponding blocks. But this implies that their line of attack gives the corresponding line of attack between the two queens at the respective positions on the initial $A\times A$ board. As  $H$ is an admissible configuration, we altogether conclude that the configuration $H[B]$ is also admissible.
\end{proof}

\begin{corollary}
\label{substopt}
If $3\mid A$ and $H$ is an admissible configuration on an $A\times A$ board with $\frac{4A}3$ queens, then $H[B]$ is
optimal on the $AB\times AB$ board.
\end{corollary}

\begin{proof}
As there are $\frac{4A}3$ queens in $H$, and as each of them is replaced by $B$ new queens in $H[B]$, there are $\frac{4AB}3$ queens in $H[B]$, which means that the configuration $H[B]$ is optimal.
\end{proof}

For us the embedment is particularly useful only when the small configuration stays as far as possible from the main diagonal, because the empty strip of diagonals that this creates
is what makes the extensions of Subsection \ref{ss32} possible. Note that inside its own block a queen stays rather close to its main
diagonal. Indeed, $p(i)$ equals either $2i$ or $2i-B$, so that $i-p(i)$ equals
$-i$ or $B-i$; and as $i$ runs from $1$ to $B$ these values run through
$-\frac{B-1}2,\dots,\frac{B-1}2$, each of them once. In other words, the $B$
queens of a block occupy the $B$ diagonals nearest to the main diagonal of the
block, one queen on each. The following lemma evaluates the clearance of the configuration $H_A[B]$.

\begin{lemma}
\label{clearance}
Let $H_A$ be a host and $d$ its clearance. Let $B$ be a positive integer with $\gcd(B,6)=1$. Then
every queen $(r,c)$ of $H_A[B]$ satisfies $|r-c|\geqslant M(d,B)$, where
\[
M(d,B)=dB-\frac{B-1}2.
\]
In other words, the clearance of $H_A[B]$ is at least $M(d,B).$
\end{lemma}

\begin{proof}
The main diagonal of the block $(R,C)$ consists of the cells with difference
$B(R-C)$, so by the preceding paragraph the difference of a queen of that block
misses $B(R-C)$ by at most $\frac{B-1}2$. It remains to recall that
$|R-C|\geqslant d$.
\end{proof}

Figure \ref{substfig} shows $H_6$ next to $H_6[5]$, an optimal configuration on
the $30\times30$ board with $40$ queens whose five central diagonals are empty,
in accordance with Lemma \ref{clearance}.

\begin{figure}[htb]
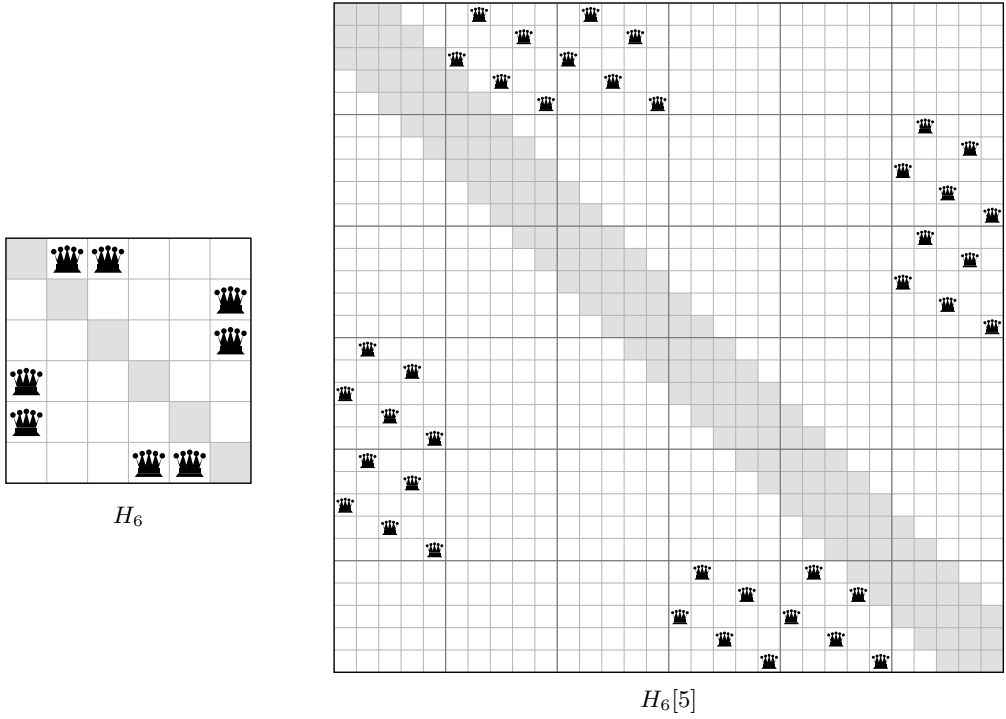

\begin{center}
\figembedment
\end{center}
\caption{The host $H_6$ and the configuration $H_6[5]$; the heavy lines mark
the $5\times5$ blocks.}
\label{substfig}
\end{figure}

\subsection{Extensions in the corner}
\label{ss32}

The embedment produces boards of side $AB$ only. To reach the sides between such ``checkpoints'' we enlarge a configuration by a few rows and columns and fill the new
part with a small configuration of its own. This works as long as the
two parts keep out of each other's way, and the empty strip of diagonals
provided by Lemma \ref{clearance} is exactly what makes this possible; see
Figure \ref{cornerfig}. The following proposition, the proof of which is omitted as it is obvious, formalizes this mechanism.

\begin{proposition}
\label{corner}
Let $M$ be a positive integer. Let $Q$ be an admissible configuration on an
$N\times N$ board all of whose queens $(r,c)$ satisfy $|r-c|\geqslant M$, and let $D$ be
an admissible configuration on a $t\times t$ board all of whose queens $(x,y)$ satisfy
$|x-y|<M$.
Move $D$ into the bottom right corner of an $(N+t)\times(N+t)$ board, that is,
replace each of its queens $(x,y)$ by $(N+x,N+y)$. Together with $Q$, this
gives an admissible configuration on the $(N+t)\times(N+t)$ board.
\end{proposition}

\begin{figure}[htb]
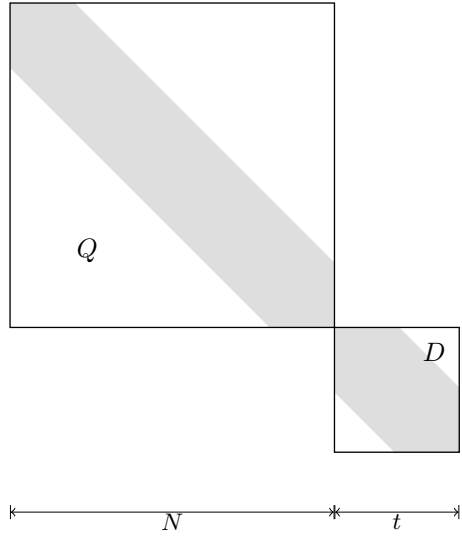

\begin{center}
\figcorner
\end{center}
\caption{The corner extension. The shaded strip consists of the diagonals with
$|r-c|<M$: it carries no queen of $Q$, and it contains every queen of $D$.}
\label{cornerfig}
\end{figure}

We shall apply this with $Q=H_A[B]$ where the host $H_A$ has clearance $d$, so that $N=AB$ and
$M=M(d,B)$. If the number $\lfloor4t/3\rfloor$ of queens we want in the corner
is attainable there, and if the corner configuration fits inside the strip
$|x-y|<M$, then the enlarged configuration is again optimal. Indeed, its two
parts contribute $\frac{4N}3$ and $\lfloor\frac{4t}3\rfloor$ queens, and since
$\frac{4N}3$ is an integer, the total is $\lfloor\frac{4(N+t)}3\rfloor$. Three
remarks settle most of the values of $t$.

First, the smallest values need no search: for $t=1$ we add the single queen
$(N+1,N+1)$, and for $t=2$ the two queens $(N+1,N+1)$ and $(N+2,N+2)$.

Second, if $6\leqslant t\leqslant M$, then any optimal configuration on a
$t\times t$ board will do, since all its queens satisfy $|x-y|\leqslant t-1<M$
automatically.

Third, the two remaining kinds of $t$ are genuinely different, and each gets
a subsection below. For $t\in\{3,4,5\}$ no corner configuration can help,
because $q(t)=t$ and thus $q(t)<\lfloor4t/3\rfloor$ there; the enlargement
has to be arranged differently. For $t>M$ an arbitrary optimal configuration on
a $t\times t$ board need not fit inside the strip, and a specially chosen one
is required.

\subsection{Splitting the new columns}
\label{ss33}

Suppose we want to enlarge the board by $t$ rows and columns, where
$t\in\{3,4,5\}$. Instead of appending all the new columns on the right, we
append $L$ of them on the left and the other $t-L$ on the right, keeping all
$t$ new rows at the bottom; the old configuration is thereby shifted $L$
columns to the right. Up to some conditions (to be formalized in the proposition that follows), this leaves enough room for $\lfloor4t/3\rfloor$ new queens
instead of merely $t$. We use $L=1,2,2$ for $t=3,4,5$ and the following new
queens, $N$ being the side of the old board:
\[
\begin{array}{|c|c|@{\,}l@{\,}|}
\hline
t&L&\text{the new queens}\\\hline
3&1&(N+1,1),\ (N+2,1),\ (N+3,N+2),\ (N+3,N+3)\\
4&2&(N+1,N+3),\ (N+2,1),\ (N+2,2),\ (N+3,N+4),\ (N+4,N+4)\\
5&2&(N+1,N+3),\ (N+1,N+4),\ (N+2,2),\ (N+3,1),\ (N+4,N+5),\ (N+5,N+5)\\[0.4mm]\hline
\end{array}
\]
There are $4$, $5$, and $6$ of them, that is, $\lfloor4t/3\rfloor$; see
Figure \ref{splitfig}.

\begin{figure}[htb]
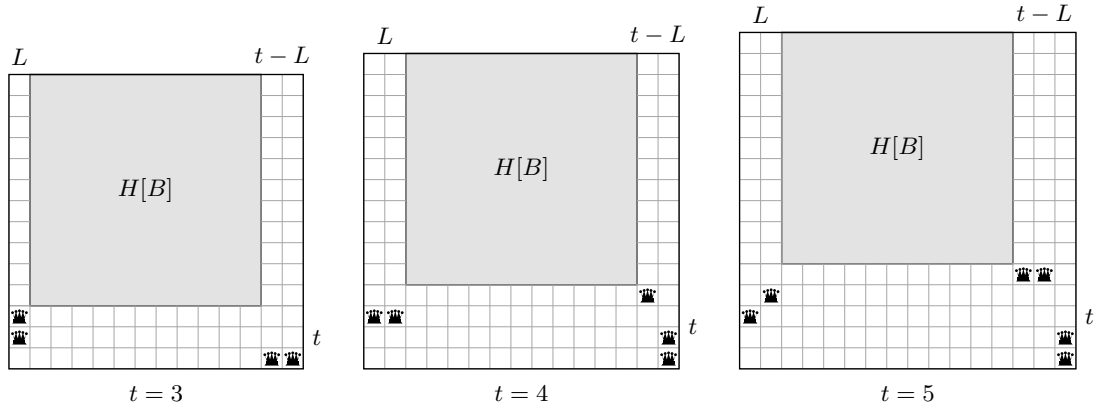

\begin{center}
\figsplit
\end{center}
\caption{The split extension for $t=3,4,5$.}
\label{splitfig}
\end{figure}

This idea was extensively used by Makay \cite{makay}, but unfortunately, not in full generality. Namely, he was considering only the case $L=1$ (he even mentioned the possibility of taking a larger $L$, but argued that this could bring no advantage; as we shall see, allowing $L=2$ can actually have an important impact). Though he \emph{ad-hoc} generated such extensions up to $t=28$, the bottleneck was the fact that such extensions were missing for $t\in\{4,6,7\}$ (he actually found them to be impossible, for the form he used). All this substantially limited the efficiency of the resulting construction. Our extension for $t=3$ is the same as Makay's extension, but for the cases $t=4$ and $t=5$ we go to $L=2$, and it turns out that this is all we need of this form: no extension with $t>5$ will be used. (One more clarification: though there exists an extension for $t=5$ and $L=1$, it is of limited usability to us; the one that we use, with $L=2$, is more productive.)

\begin{proposition}
\label{split}
Let $3\mid A$. Let $H$ be an admissible configuration on an $A\times A$ board with $\frac{4A}3$ queens, having no queen on
either of the two main diagonals. Let $B$ be an integer with $B\geqslant5$ and
$\gcd(B,6)=1$, let $t\in\{3, 4, 5\}$, and put $N=AB$. Then the queens
of $H[B]$, shifted $L$ columns
to the right, together with the new queens listed above, form an optimal
configuration on the $(N+t)\times(N+t)$ board.
\end{proposition}

\begin{proof}
The number of queens is right, so only admissibility is at issue. Old and new
queens occupy disjoint sets of rows, and also disjoint sets of columns, so only
differences and sums can cause trouble.

Among the new queens there is nothing to check beyond what the table shows: for
$t=3$ a pair in a column and a pair in a row; for $t=4$ a pair in a row, a pair
in a column, and one queen meeting nobody; for $t=5$ a pair in a row, a pair on
an antidiagonal, and a pair in a column. Apart from these occasions, the four
numbers of the new queens are pairwise different, which one reads off the table
at a glance.

It remains to check that an old and a new queen never attack each other. Clearly, this is never the case for horizontal and vertical attacks, so it remains to take care of diagonal attacks.

Note that all the new queens in the right-hand columns belong to the extension of either the main diagonal of $H[B]$, or a diagonal at distance at most $2$ from the main diagonal. As the clearance of $H$ is at least $1$ (since it was assumed that $H$ has no queens on the main diagonal), Lemma \ref{clearance} gives that the clearance of $H[B]$ is at least $3$. Therefore, we conclude that none of the new queens in the right-hand columns attack an old queen.

Regarding the new queens in the left-hand columns, we note the following. The main antidiagonal of $H[B]$ is empty, as well as the antidiagonal immediately below it. The first claim follows as the main antidiagonal of $H$ is empty (by the assumption), and the second claim follows by the same observation, together with additionally acknowledging that the cell $(1,1)$ is empty in our $B\times B$ board. This completes the proof.
\end{proof}

Of our four hosts, $H_6$, $H_{12}$, and $H_{18}$ have both main diagonals empty,
as recorded in the beginning of Section \ref{s3}, and Proposition \ref{split} applies to them.
The host $H_{15}$ does not qualify, and we cannot simply replace it by a better
$15\times15$ host, since its clearance $d=2$ is needed elsewhere. Instead we
keep a second $15\times15$ configuration for this one purpose, namely the one drawn in Figure \ref{sfifteenfig}. It is optimal, and both of its main
diagonals are empty, so Proposition \ref{split} does apply to it. Its clearance is only $1$, and this is why the two
$15\times15$ configurations are kept side by side: $H_{15}$ serves the
extensions of Proposition \ref{corner}, and $H_{15}'$ those of
Proposition \ref{split}. (Of course, it would be nice to have a unique $15\times 15$ configuration that serves both purposes: that is, that has clearance $2$, and empty main antidiagonal. However, such a configuration does not exist, we checked exhaustively.)

\begin{figure}[htb]
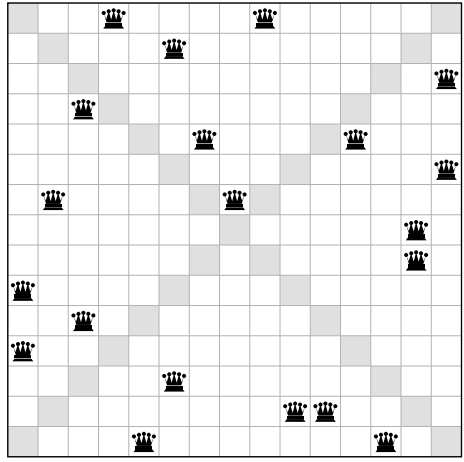

\begin{center}
\figsfifteen
\end{center}
\caption{The configuration $H_{15}'$. The shaded cells form the two main
diagonals, which Proposition \ref{split} requires to be empty.}
\label{sfifteenfig}
\end{figure}

\subsection{Four corner configurations close to the diagonal}
\label{ss42}

Proposition \ref{corner} asks the corner configuration to satisfy $|x-y|<M$,
and for $t>M$ this has to be arranged by hand. Only the values
$8\leqslant t\leqslant11$, with $M\geqslant7$, will occur in Section \ref{s5},
and for these we use the configurations $C_8,C_9,C_{10},C_{11}$
drawn in Figure \ref{cornersfig}. They have $10$, $12$, $13$, and $14$ queens
(that is, $\lfloor4t/3\rfloor$ in each case), they are admissible, and each of
them satisfies $|x-y|\leqslant6$. So by Proposition \ref{corner} any of them
may be put in the corner of a configuration whose clearance $M$ is at least
$7$. Of the four, $C_8$ is the $8\times8$ configuration of \cite{resta} and
$C_9$ is its $9\times9$ configuration reflected; the $10\times10$ and
$11\times11$ configurations displayed there have queens with $|x-y|=8$ and
would not do.

\begin{figure}[htb]
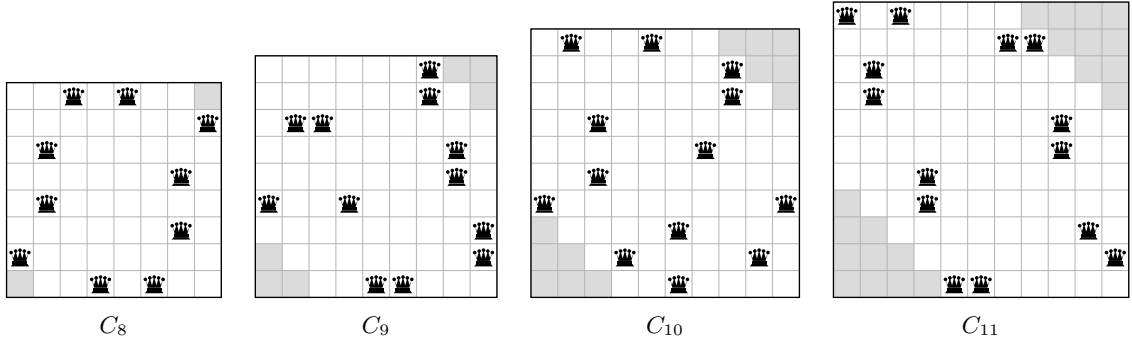

\begin{center}
\figcorners
\end{center}
\caption{The corner configurations $C_8$, $C_9$, $C_{10}$, and $C_{11}$. The
shaded cells are those with $|x-y|\geqslant7$.}
\label{cornersfig}
\end{figure}

\section{Proof of the main theorem}
\label{s5}

It remains to produce an optimal configuration on an $n\times n$ board for
every $n$ from $73$ onward. Everything we have prepared is summarized in the
following lemma, which says what one block of the construction yields: from a
host $H_A$ and a multiplier $B$ we obtain a whole interval of board sides,
beginning at $AB$.

\begin{lemma}
\label{block}
Let $A\in\{6, 12, 15, 18\}$, let $d$ be the clearance of the host
$H_A$, and let $B$ be an integer with $B\geqslant5$ and $\gcd(B,6)=1$. Put
$N=AB$ and $M=M(d,B)$, and assume $M\geqslant7$. If $q(t)=\lfloor4t/3\rfloor$ for every $t$ with
$6\leqslant t\leqslant M$, then
\[
q(n)=\left\lfloor\frac{4n}{3}\right\rfloor
\qquad\text{for every $n$ with $N\leqslant n\leqslant N+\max\{M,11\}$.}
\]
\end{lemma}

\begin{proof}
Write $n=N+t$. By Corollary \ref{substopt} and Lemma \ref{clearance}, the
configuration $H_A[B]$ is optimal on the $N\times N$ board and all its queens
satisfy $|r-c|\geqslant M$, so Proposition \ref{corner} may be applied to it
with any corner configuration lying inside the strip $|x-y|<M$. For $t=0$ there
is nothing to add, and $t=1$, $t=2$ were dealt with near the end of
Subsection \ref{ss32}.

For $t\in\{3,4,5\}$ we use Proposition \ref{split} instead, with $H_A$ if
$A\neq15$ and with $H_{15}'$ if $A=15$. For $6\leqslant t\leqslant M$ we put an
optimal configuration on a $t\times t$ board in the corner; by assumption it
has $\lfloor4t/3\rfloor$ queens, and it fits inside the strip. There remain the values of
$t$ with $M<t\leqslant11$; as $M\geqslant7$, such a $t$ satisfies
$8\leqslant t\leqslant11$, and we put $C_t$ in the corner, which is allowed
because, for any queen $(x,y)$ on $C_t$, we have $|x-y|\leqslant6<7\leqslant M$.

In each case the resulting configuration is admissible and has
$\lfloor\frac{4n}3\rfloor$ queens, as counted after Proposition \ref{corner},
so it is optimal by Theorem \ref{upper}.
\end{proof}

\subsection{The sides $73$ and $74$}
\label{s51}

Lemma \ref{block} does not reach $73$ and $74$. Indeed, since $B\geqslant5$,
the products $AB$ not exceeding $74$ are $30$, $42$, and $66$ for $A=6$, and $60$
for $A=12$, and there are none for $A=15$ and $A=18$; the corresponding
clearances are $3$, $4$, $6$, and $8$. The first three are below $7$, so
Lemma \ref{block} is unavailable for them, and Proposition \ref{corner} by itself pushes
$66$ up to at most $72$; the last product is pushed up to at most $71$.

So we start from $H_6[11]$ (which is an optimal configuration on the
$66\times66$ board with $88$ queens), and place it inside a larger board so that
an empty frame is left around it. For $n=73$ we shift $H_6[11]$ four rows down
and four columns to the right, so that it occupies the rows and the columns
$5,6,\dots,70$, and add the nine queens
\[
\{(1,2),(1,71),(2,4),(3,1),(4,3),(4,72),(71,4),(72,1),(73,73)\},
\]
all of which lie in the frame. For $n=74$ we shift
$H_6[11]$  four rows down and five columns to the right and add the ten queens
\[
\{(1,4),(2,72),(3,74),(4,3),(4,5),(71,2),(71,73),(72,4),(73,1),(74,74)\}.
\]
This is illustrated in Figure \ref{fig73}.
Both configurations are admissible, as a direct check 
shows. The first one has $97$ queens and the second one $98$, which are the respective
values of $\lfloor4n/3\rfloor$; hence $q(73)=97$ and $q(74)=98$.

\begin{remark}
    Basically, what we did here, is that our $H_6[11]$ configuration is extended to the size $+7$, respectively $+8$. This is not a specific property of $H_6[11]$; the same patches work for some other boards too, and furthermore, similar patches can be designed to obtain $+6$, $+9$, $+10$, and $+11$ extensions of various boards. Everything together can handle all the values from $n=31$ onward, with the following exceptions: $n\in\{39, 40, 54, 55, 56, 59\}$. As explained in Introduction, once the table \cite{healy} appeared, we decided to abandon this path and instead present the work in a more clear manner, settling the cases from $n=73$ onward.
\end{remark}

\begin{figure}[htb]
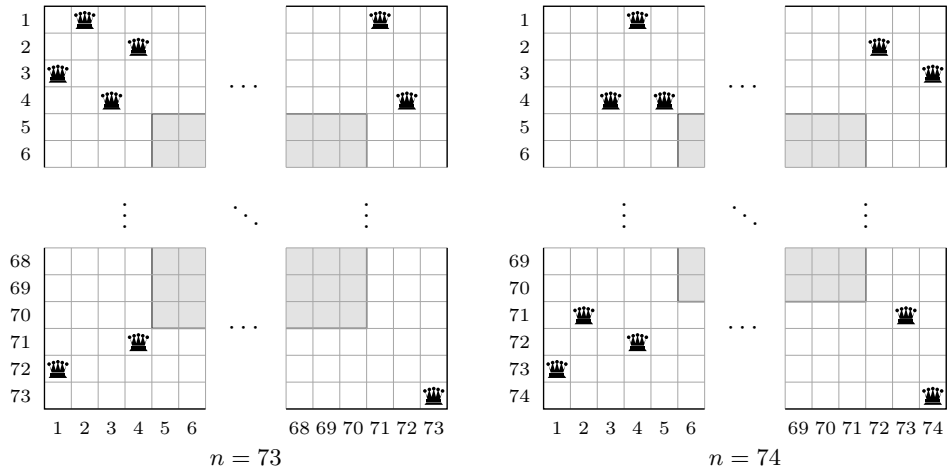

\begin{center}
\figbigboard
\end{center}
\caption{Optimal configurations on the $73\times73$ and $74\times74$ boards; only the four corners of each board are drawn. The
shaded cells are those of the shifted copy of $H_6[11]$.}
\label{fig73}
\end{figure}

\subsection{The range $75\leqslant n\leqslant281$}
\label{s52}

Here we handle what the subsection title says.

\begin{lemma}
\label{finite}
We have $q(n)=\lfloor4n/3\rfloor$ for every $n$ with $75\leqslant n\leqslant281$.
\end{lemma}

\begin{proof}
Each row of the table below is one application of Lemma \ref{block}; the last
column is the interval $[AB,\ AB+\max\{M,11\}]$ that the row delivers.
\[
\begin{array}{|c|c|c|c|l|}
\hline
A&B&N=AB&M&\ \text{sides obtained}\\\hline
15&5&75&8&\ 75\text{--}86\\
6&13&78&7&\ 78\text{--}89\\
18&5&90&8&\ 90\text{--}101\\
6&17&102&9&\ 102\text{--}113\\
6&19&114&10&\ 114\text{--}125\\
18&7&126&11&\ 126\text{--}137\\
12&11&132&17&\ 132\text{--}149\\
6&25&150&13&\ 150\text{--}163\\
12&13&156&20&\ 156\text{--}176\\
6&29&174&15&\ 174\text{--}189\\
6&31&186&16&\ 186\text{--}202\\
18&11&198&17&\ 198\text{--}215\\
12&17&204&26&\ 204\text{--}230\\
6&37&222&19&\ 222\text{--}241\\
18&13&234&20&\ 234\text{--}254\\
15&17&255&26&\ 255\text{--}281\\\hline
\end{array}
\]
Every $B$ occurring here satisfies $\gcd(B,6)=1$ and $B\geqslant5$, and every
$M$ is at least $7$, so Lemma~\ref{block} is indeed applicable, provided that
$q(t)=\lfloor4t/3\rfloor$ for the relevant auxiliary sides $t$. These satisfy
$6\leqslant t\leqslant26$, and for such $t$ the required configurations are
listed in \cite{resta}.

It remains to observe that the sixteen intervals in the last column cover
$\{75,76,\dots,281\}$, which they do: each of them begins before or immediately after the previous
one ends.
\end{proof}

\subsection{The sides $n\geqslant282$}
\label{s53}

From here on the host $H_6$ alone suffices, as the intervals produced by
Lemma \ref{block} leave no gaps. This is more formally presented in the following two lemmas.

\begin{lemma}
\label{intervals}
Every integer $n$, $n\geqslant282$, lies in one of the intervals
\[
[36k-6,\ 39k-6],\qquad [36k+6,\ 39k+7]\qquad (\text{for $k\geqslant8$}),
\]
which are the intervals delivered by Lemma \ref{block} for $A=6$ and
$B=6k-1$, respectively $B=6k+1$.
\end{lemma}

\begin{proof}
Both choices of $B$ are coprime to $6$. For $B=6k-1$ we have $N=6B=36k-6$ and
\[
M(1,B)=B-\frac{B-1}2=\frac{B+1}2=3k,
\]
which for $k\geqslant8$ exceeds $11$; so Lemma \ref{block} delivers the sides
from $36k-6$ to $39k-6$. For $B=6k+1$ we get $N=36k+6$ and $M(1,B)=3k+1$,
hence the sides from $36k+6$ to $39k+7$.

Consecutive intervals of the sequence
$[36\cdot8-6,\,39\cdot8-6],\ [36\cdot8+6,\,39\cdot8+7],\ [36\cdot9-6,\,39\cdot9-6],\dots$
overlap or immediately follow one another, because for $k\geqslant8$ we have
\[
(39k-6)+1\geqslant36k+6\qquad\text{and}\qquad(39k+7)+1\geqslant36(k+1)-6;
\]
indeed, these inequalities are equivalent to $3k\geqslant11$ and $3k\geqslant22$. The first
interval of the sequence is $[282,306]$, so their union is $[282,\infty)$.
\end{proof}

\begin{lemma}
\label{tail}
We have $q(n)=\lfloor4n/3\rfloor$ whenever $n\geqslant282$.
\end{lemma}

\begin{proof}
We argue by strong induction on $n$. By the arguments up to Subsection \ref{s52}, the assertion holds for every side between $6$ and
$281$.

Let $n\geqslant282$ and assume that the assertion holds for all smaller sides. Choose
$k\geqslant8$ and $B=6k\pm1$ such that $n$ belongs to the corresponding interval of
Lemma \ref{intervals} (or to any such one, if there is a choice), and put $M=M(1,B)$, which is $3k$ or $3k+1$. As
$B\geqslant47$ and $M\geqslant24$, Lemma \ref{block} may be applied to $A=6$
and this $B$, provided that $q(t)=\lfloor4t/3\rfloor$ for $6\leqslant t\leqslant M$;
and this is guaranteed by the induction hypothesis, since $M<36k-6\leqslant n$.
\end{proof}

Finally, we are ready to sum up everything.

\begin{proof}[Proof of Theorem \ref{main}]
The upper bound $q(n)\leqslant\lfloor4n/3\rfloor$ is Theorem \ref{upper}, and
$q(n)=n$ for $n\leqslant5$ was noted in Section \ref{s2}. The bound is attained
for $6\leqslant n\leqslant72$ by the configurations of \cite{healy}, for
$n=73$ and $n=74$ by the two configurations of Subsection \ref{s51}, for
$75\leqslant n\leqslant281$ by Lemma \ref{finite}, and for $n\geqslant282$ by
Lemma \ref{tail}.
\end{proof}

\section*{Acknowledgments}

The authors were supported by the Ministry of Science, Technological Development and Innovation of the Republic of Serbia (grants no. 451-03-33/2026-03/200125 and 451-03-34/2026-03/200125 for the first two authors, and grant no. 451-03-34/2026-03/200156 for the third author).

\sloppy

\end{document}